\documentclass[10pt,leqno,letterpaper]{article}
\usepackage[utf8]{inputenc}
\usepackage{amsmath}
\usepackage{amsfonts}
\usepackage{amssymb}
\usepackage{graphicx}
\usepackage{mathrsfs}
\usepackage{upref,amsthm,amsxtra,exscale}
\usepackage{cite}
\usepackage[colorlinks=true,urlcolor=blue,
citecolor=red,linkcolor=blue,linktocpage,pdfpagelabels,
bookmarksnumbered,bookmarksopen]{hyperref}
\usepackage{upgreek}

\usepackage{fullpage}

\newtheorem{theorem}{Theorem}[section]

\newtheorem{remark}[theorem]{Remark}
\newtheorem{lemma}[theorem]{Lemma}
\newtheorem{proposition}[theorem]{Proposition}

\numberwithin{equation}{section}

\def\r{\mathbb{R}}
\def\s{\mathbb{S}}
\def\rn{\mathbb{R}^N}
\def\z{\mathbb{Z}}

\def\eps{\varepsilon}

\def\io{\int_{\Omega}}
\def\irn{\int_{\r^N}}

\def\o{\Omega}
\def\t{\Theta}
\def\bf{\mathbf}
\def\tilde{\widetilde}
\def\cA{\mathcal{A}}

\def\cC{\mathcal{C}}

\def\cJ{\mathcal{J}}

\def\cM{\mathcal{M}}
\def\cN{\mathcal{N}}

\def\cT{\mathcal{T}}
\def\cU{\mathcal{U}}

\def\cat{\mathrm{cat}}
\def\dist{\mathrm{dist}}

\def\r{\mathbb{R}}
\def\rn{\mathbb{R}^N}

\def\z{\mathbb{Z}}

\def\eps{\varepsilon}

\def\io{\int_{\Omega}}
\def\irn{\int_{\r^N}}

\def\tilde{\widetilde}

\def\cJ{\mathcal{J}}
\def\cM{\mathcal{M}}
\def\cN{\mathcal{N}}

\def\cU{\mathcal{U}}

\def\cZ{\mathcal{Z}}
\def\bu{\mathbf{u}}

\newcommand{\B}[1]{{\color{blue} #1}}  

\author{Mónica Clapp, \ Alberto Saldaña, \ and \ Andrzej Szulkin}
\title{Configuration spaces and multiple positive solutions to a singularly perturbed elliptic system}
\date{}

\begin{document}
\maketitle
	
\begin{abstract}
We consider a weakly coupled singularly perturbed variational elliptic system in a bounded smooth domain with Dirichlet boundary conditions.  We show that, in the competitive regime, the number of fully nontrivial solutions with nonnegative components increases with the number of equations. Our proofs use a combination of four key elements: a convenient variational approach, the asymptotic behavior of solutions (concentration), the Lusternik-Schnirelman theory, and new estimates on the category of suitable configuration spaces.  

\medskip

\medskip


\end{abstract}

\section{Introduction}

We consider the following system of singularly perturbed elliptic equations
\begin{equation} \label{eq:system}
\begin{cases}
-\eps^2\Delta u_i+u_i=\sum\limits_{j=1}^\ell\beta_{ij}|u_j|^p|u_i|^{p-2}u_i,\\
u_i \in H^1_0(\o), \quad u_i\neq 0, \qquad i=1,\ldots,\ell,
\end{cases} 
\end{equation}
where $\eps>0$ is a small parameter, $\o$ is a bounded smooth domain in $\rn$, $N\geq 2$, $\ell\geq 2$, $\beta_{ii}>0$, $\beta_{ij}=\beta_{ji}<0$ if $i\neq j$, and $2p\in (2,2^*)$, where $2^*$ is the critical Sobolev exponent, i.e., $2^*:=\infty$ if $N=2$ and $2^*:= \frac{2N}{N-2}$ if $N\geq 3$.

System \eqref{eq:system} is used in physics to model a variety of phenomena; for instance, it is a model for a binary mixture of Bose-Einstein condensates in two different hyperfine states (then $\ell=2$) and in nonlinear optics. It also serves as a model in population dynamics. In this paper we assume $\beta_{ii}>0$ and $\beta_{ij}<0$ for $i\neq j$ which means that the interaction between species (states) of the same type is \emph{attractive} while it is \emph{repulsive} for species of different types. However, our point of view in this article is purely mathematical. 

A solution $\bf u=(u_1,\ldots,u_\ell)$ to \eqref{eq:system} is called \emph{fully nontrivial} if $u_i\neq 0$ for every $i=1,\ldots,\ell$, and we shall call it \emph{nonnegative} if, in addition, $u_i\geq 0$ for all $i=1,\ldots,\ell$.

Most of the research for \eqref{eq:system} is focused on the case $\ell=p=2$ and $N=1,2,3$. In this setting, the asymptotic behavior (as $\eps\to 0$) of nonnegative solutions of \eqref{eq:system} was studied in detail in the seminal paper \cite{lw}. There, it is shown that the least energy solutions exhibit concentration. To be more precise, as $\eps\to 0,$ the $i$-th component is close to a rescaling and translation of the positive radially symmetric ground state solution of
\begin{align*}
-\Delta u+u=\beta_{ii}|u|^{2p-2}u,\qquad
u \in H^1(\rn).
\end{align*}
The concentration points approach a configuration that maximizes the distance between them and to $\partial\o$.  Furthermore, the existence of \emph{two} nonnegative solutions for $\ell=p=2$ and $N=1,2,3$ is shown in \cite{WWL} using the Lusternik-Schnirelman theory.
 
This multiplicity result is interesting when compared to the case of a single equation. Consider, for example,
\begin{align}\label{1eq}
-\Delta u+u=|u|^{2p-2}u,\qquad
u \in H_0^1(\Omega)\smallsetminus\{0\}.
\end{align} 
Uniqueness or multiplicity results for \eqref{1eq} rely on both the geometry and the topology of the domain $\Omega$ (see the introduction of \cite{dis23} for an updated survey in this regard).  In particular, if $\Omega$ is a ball, then \eqref{1eq} has a unique positive solution. 
 
Therefore, a natural question is whether,  for any domain, the number of nonnegative fully nontrivial solutions of \eqref{eq:system} increases as the number of equations $\ell$ becomes larger. We give a positive answer to this question.  Our main result is the following one. 

\begin{theorem}\label{m:thm}
For $\eps$ small enough, the system \eqref{eq:system} has at least $\ell$ nonnegative solutions.
\end{theorem}

Note that this result concerns any dimension $N\geq 2$. If $p<2$, which is necessarily the case for $N\ge 4$, then neither the functional corresponding to \eqref{eq:system} nor the Nehari-type manifold (defined later) is of class $\cC^2$. It will therefore be convenient to employ the method of \cite{csz}. Then, we follow the ideas introduced in \cite{bc} which allow to estimate the number of critical points in the presence of concentration using the Lusternik-Schnirelman theory. For a multiplicity result in dimensions $N=2$ and $N=3$ with $p=2$, under symmetry assumptions on the domain and the coupling coefficients, we refer to \cite{Huang}. 
	
Theorem~\ref{m:thm} is a direct consequence of Theorems~\ref{thm:main} and~\ref{thm:configuration} below. To state these results we introduce some notation.  

Let $X$ be a topological space. Recall that \emph{the Lusternik-Schnirelman category of a subset $A$ of $X$ in $X$}, denoted $\cat(A\hookrightarrow X)$ (or $\cat_X(A)$), is the the smallest number of subsets of $A$ that cover $A$ and each of them is open in $A$ and contractible in $X$. If $A=X$ we write $\cat(X)$ instead of $\cat(X\hookrightarrow X)$.

Let $\t$ be a subset of $\rn$. For $r\geq 0$ and $\ell\geq 2$, we consider the space
$$F_{\ell,r}(\t):=\{(\xi_1,\ldots,\xi_\ell)\in\t^\ell:\dist(\xi_i,\rn\smallsetminus\t)>r\text{ \ and \ }|\xi_i-\xi_j|> 2r\text{ if }i\neq j\},$$
endowed with the subspace topology of $\t^\ell$. If $r=0$ we write $F_\ell(\t):=F_{\ell,0}(\t)$. This last space is called the \emph{ordered configuration space of $\ell$ points in $\t$}. Finally, we set
$$\t^+_r:=\{x\in\rn:\dist(x,\t)<r\}.$$

The following theorem gives a lower bound on the number of nonnegative solutions of \eqref{eq:system}.

\begin{theorem} \label{thm:main}
Given $r>0$, there exists $\eps_r>0$ such that the system \eqref{eq:system} has at least 
\begin{align}\label{cat}
\cat(F_{\ell,r}(\o)\hookrightarrow F_\ell(\o_r^+)) 
\end{align}
nonnegative solutions for any $\eps\in(0,\eps_r)$.
\end{theorem}

The next result gives an estimate for \eqref{cat}.

\begin{theorem} \label{thm:configuration}
For $r>0$ sufficiently small, we have that
$$\cat(F_{\ell,r}(\o)\hookrightarrow F_\ell(\o_r^+))\geq\ell.$$
Furthermore, if $\o$ is convex and $r>0$ is sufficiently small, then \ $\cat(F_{\ell,r}(\o)\hookrightarrow F_\ell(\o))=\ell$.
\end{theorem}

Let us briefly discuss the strategy of the proofs. Existence of a nonnegative solution $\bf{u}=(u_1,\ldots,u_\ell)$ was established in \cite{cs} by minimization on a Nehari-type manifold following the approach in \cite{csz}. Denote such manifold corresponding to our system \eqref{eq:system} by $\cN_{\eps}(\o)$. We show that all low-energy solutions, i.e. solutions $\bu\in \cN_{\eps}(\o)^{\leq \overline d}$ for a suitable $\overline d$, do not change sign and exhibit concentration as $\eps\to 0$ (see  Section \ref{sec:variational} for the definitions, and Theorem~\ref{thm:profiles} and Proposition~\ref{prop:sign} for more rigorous statements).  To count only the \emph{nonnegative} solutions, we use a suitable quotient space. Let $\z_2:=\{1,-1\}$ and consider the group $\cZ:=(\z_2)^\ell$ acting on $H_0^1(\o)^\ell$ by $\bf s\bf u:=(s_1u_1,\ldots,s_\ell u_\ell)$ for $\bf s=(s_1,\ldots,s_\ell)\in\cZ$ and $\bf u=(u_1,\ldots,u_\ell)\in H_0^1(\o)^\ell.$  Then, applying the method of \cite{csz} and some standard results in critical point theory with symmetries we show that \eqref{eq:system} has at least $\cat(\cN_{\eps}(\o)^{\leq\overline d}/\cZ)$ nonnegative solutions (see Theorem~\ref{thm:cat}), where $\cN_{\eps}(\o)^{\leq\overline d}/\cZ$ is the space of $\cZ$-orbits in $\cN_{\eps}(\o)^{\leq\overline d}$. Now the task is to explicitly estimate this number. Using the concentration behavior, we construct maps
\begin{align}\label{i:intro}
F_{\ell,r}(\o)\xrightarrow{\bf i_\eps}\cN_{\eps}(\o)^{\leq \overline{d}}\xrightarrow{q}\cN_{\eps}(\o)^{\leq \overline{d}}/\cZ\xrightarrow{\widetilde{\bf b}}F_\ell(\o_r^+),
\end{align}
where $\bf i_\eps:F_{\ell,r}(\o)\to \cN_{\eps}(\o)^{\leq \overline d}$ (defined in \eqref{i:def}) places a copy of a suitable concentrated profile around each point of an $\ell$-tuple $(\xi_1,\ldots,\xi_\ell)\in F_{\ell,r}(\o)$, $q:\cN_{\eps}(\o)^{\leq \overline{d}}\to\cN_{\eps}(\o)^{\leq \overline{d}}/\cZ$ is the quotient map, and $\widetilde{\bf b}$ is defined by $\bf b=\widetilde{\bf b}\circ q$, where $\bf b:\cN_{\eps}(\o)^{\leq d}\to F_\ell(\o_r^+)$ is a generalized barycenter map (see Proposition~\ref{prop:b}). The composition of these maps is the inclusion, whose category is a lower bound for $\cat\left(\cN_{\eps}(\o)^{\leq d}/\cZ\right)$ (see Section~\ref{sec:cs}).  

Theorem~\ref{thm:configuration} is proved by induction on $\ell$ using suitably constructed fibrations which allow to apply some basic tools from algebraic topology in order to obtain the estimate $\cat(F_{\ell,r}(\o)\hookrightarrow F_\ell(\o_r^+))\geq \cat (F_\ell(\rn)) = \ell$ (see Section~\ref{sec:configuration}). The equality $\cat (F_\ell(\rn)) = \ell$ is shown in \cite[Theorem 1.2]{r}.

Therefore, methodologically, our main contribution is to show how the category of $F_{\ell,r}(\o)\hookrightarrow F_\ell(\o_r^+)$ can be estimated and used to prove multiplicity of nonnegative solutions of variational systems that exhibit concentration. We believe that this approach can be useful in other problems as well, for instance, to study systems of equations with suitable potentials in unbounded domains, and to obtain multiplicity results for low-energy nodal solutions of \eqref{eq:system}.

Our results hold true for systems with more general power nonlinearities, like those considered in \cite{cs,csz,css}. We have considered the particular system \eqref{eq:system} for the sake of simplicity.

System \eqref{eq:system} has been extensively studied in recent years from several perspectives. Here we just mention some of these results to provide an overview on different aspects of the problem.  For a study of dependence of the concentration points on suitable potentials, see \cite{p06}. For results on sign-changing solutions we refer to \cite{cs} and the references therein. For problems in $\rn$, see \cite{si,BjLyMsh,WjWy}, and for a recent study on the decay at infinity of solutions we refer to \cite{acs23}. For a relationship between nodal solutions and fully nontrivial solutions of \eqref{eq:system} in the critical case, see \cite{css}.

We also mention that the use of configuration spaces, barycenter maps, and the Lusternik-Schnirelman theory is not completely new.  It has been used, for instance, in \cite{bcw} to study lower bounds for the number of nodal solutions to a singularly perturbed nonlinear Schrödinger equation. 

The paper is organized as follows. In Section~\ref{sec:variational} we present the variational framework and give a lower bound for the number of solutions to \eqref{eq:system} in terms of $\cat\left(\cN_{\eps}(\o)^{\leq d}/\cZ\right)$. Section~\ref{sec:sign} is devoted to showing that low-energy solutions do not change sign. The maps \eqref{i:intro} are constructed in Section~\ref{sec:cs}, where we also give a proof of Theorem~\ref{thm:main}. Finally, in Section~\ref{sec:configuration}, we show the estimate in Theorem~\ref{thm:configuration}.

\section{The variational problem}
\label{sec:variational}

For $\eps>0$ fixed, let 
\[
\|v\|_\eps := \left(\frac1{\eps^N}\io(\eps^2|\nabla v|^2+v^2)\right)^{1/2}
\] 
be the norm of $v$ in $H^1_0(\o)$ (equivalent to the usual one).
We write the elements in $H_0^1(\o)^\ell$ as $\bf u=(u_1,\ldots,u_\ell)$ and we introduce the norm $\|\bf u\|_\eps := (\|u_1\|_\eps^2 + \cdots +\|u_\ell\|_\eps^2)^{1/2}$. The solutions to the system \eqref{eq:system} are the critical points of the functional $\cJ_\eps:H_0^1(\o)^\ell\to\r$ given by
\begin{align*}
\cJ_\eps(\bf u):= \frac{1}{2\eps^N}\sum_{i=1}^\ell\io\left(\eps^2|\nabla u_i|^2 + u_i^2\right) - \frac{1}{2p\eps^N}\sum_{i,j=1}^\ell\io\beta_{ij}|u_j|^p|u_i|^p,
\end{align*}
which is of class $\mathcal{C}^1$. Its partial derivatives at $\bf u$ are given by
\begin{align*}
\partial_i\cJ_\eps(\bf u)v =\frac{1}{\eps^N}\Big[\io(\eps^2\nabla u_i \cdot \nabla v + u_iv) -\sum_{j=1}^\ell\io\beta_{ij}|u_j|^p|u_i|^{p-2}u_iv\Big],
\end{align*}
for $v \in H^1_0(\o)$ and $i=1,\ldots,\ell$. The fully nontrivial solutions of \eqref{eq:system} belong to the Nehari-type set
\begin{equation}\label{N}
\cN_{\eps}(\o):=\Big\{\bf u\in H_0^1(\o)^\ell:u_i\neq 0, \ \partial_i\cJ_\eps(\bf u)u_i=0, \ \forall \ i = 1,\ldots,\ell\Big\}.
\end{equation}
Define 
\begin{equation*}
c_\eps(\o):= \inf_{\bf u\in\cN_{\eps}(\o)}\cJ_\eps(\bf u).
\end{equation*}
If $\bf u\in\cN_\eps(\o)$ satisfies $\cJ_\eps(\bf u)=c_\eps(\o)$, then $\bf u=(u_1,\ldots,u_\ell)$ is a solution to the system \eqref{eq:system}, see \cite[Theorem 3.4(a)]{csz}. It is called a least energy solution. Since $\o$ is bounded and $2p$ is subcritical, a standard argument yields a least energy solution for any $\eps>0$, see \cite[Theorem 2.6]{cs}. 

Let $\z_2:=\{1,-1\}$. Recall that the group $\cZ:=(\z_2)^\ell$ acts on $H_0^1(\o)^\ell$ by
\begin{equation}\label{eq:action}
\bf s\bf u:=(s_1u_1,\ldots,s_\ell u_\ell)\qquad\text{where \ }\bf s=(s_1,\ldots,s_\ell)\in\cZ\text{ \ and \ }\bf u=(u_1,\ldots,u_\ell)\in H_0^1(\o)^\ell.
\end{equation}
The \emph{$\cZ$-orbit of $\bf u$} is the set $\cZ\mathbf{u}:= \{\mathbf{s}\mathbf{u}: \mathbf{s}\in\cZ\}$. A subset $\cA$ of $H_0^1(\o)^\ell$ is called \emph{$\cZ$-invariant} if $\cZ\bf u\subset\cA$ whenever $\bf u\in\cA$, and a function $\Phi:\cA\to\r$ is called \emph{$\cZ$-invariant} if $\Phi$ is constant on each $\cZ$-orbit of $\cA$. The \emph{$\cZ$-orbit space of $\cA$} is the set $\cA/\cZ:=\{\cZ\bf u:\bf u\in\cA\}$ endowed with the quotient space topology. Note that $\cN_\eps(\o)$ and $\cJ_\eps$ are $\cZ$-invariant, and so is the sublevel set
$$\cN_{\eps}(\o)^{\leq d}:=\{\bf u\in\cN_{\eps}(\o):\cJ_\eps(\bf u)\leq d\},\qquad d\in\r.$$
If $\bf u$ is a critical point of $\cJ_\eps$, then so is $\bf s\bf u$ for every $\bf s\in\cZ$, and $\cZ\bf u$ is called a \emph{critical $\cZ$-orbit of $\cJ_\eps$}. 

\begin{theorem} \label{thm:cat}
Let $\eps>0$ and $d\in\r$. Then, $\cJ_\eps$ has at least
$$\cat\left(\cN_{\eps}(\o)^{\leq d}/\cZ\right)$$
critical $\cZ$-orbits in $\cN_{\eps}(\o)^{\leq d}$.
\end{theorem}

\begin{proof}
The proof follows by an adaptation of the argument in \cite[Theorem 3.4]{csz} and well known results in critical point theory with symmetries. For the reader's convenience we sketch it here.

Fix $\eps>0$. Let $S:=\{v\in H^1_0(\o): \|v\|_\eps=1\}$, $\mathcal{T}:=S\times\cdots\times S$ ($\ell$ times) and
\[
\cU_\eps:= \{\mathbf{u}\in \mathcal{T}: \bf t\mathbf{u}\in \cN_{\eps}(\o)\text{ for some }\bf t=(t_1,\ldots,t_\ell)\in(0,\infty)^\ell\}.
\]
For $\bu\in\mathcal{U}_\eps$ there exists a unique $\bf t_\bu\in(0,\infty)^\ell$ such that $\bf t_\bu\bu\in\cN_\eps(\o)$. Let $\mathbf{m}_\eps(\bu) := \bf t_\bu\bu$. Then $\mathbf{m}_\eps: \mathcal{U}_\eps\to\cN_\eps(\o)$ is a homeomorphism. This and more can be seen from \cite[Proposition 3.1]{csz}. Now, let
\[
\Psi_\eps(\bu) := \cJ_\eps(\bf t_\bu\bu),\qquad\bf u\in\cU_\eps.
\]
According to \cite[Theorem 3.3]{csz}, $\Psi_\eps\in \mathcal{C}^1(\cU_\eps,\r)$ and $\bu$ is a critical point of $\Psi_\eps$ if and only if $\mathbf{m}_\eps(\bu)$ is a fully nontrivial critical point of $\cJ_\eps$. Now a standard argument applies. One can construct a pseudogradient vector field for $\Psi_\eps$ on $\cU_\eps$ and obtain a deformation lemma using its flow, as e.g. in \cite[Section 5.3]{w}. By (iii) of \cite[Theorem 3.3]{csz} the flow cannot reach the boundary of $\cU_\eps$.

Since 
\[
\cJ_\eps(\bf u) - \frac1{2p}\cJ_\eps'(\bf u)\bf u =  \left(\frac12 - \frac1{2p}\right)\|\bf u\|_\eps^2,
\]
it is easy to see using the compactness of the embedding $H^1_0(\o) \hookrightarrow L^{2p}(\o)$ that 
 $\cJ_\eps$ satisfies the Palais-Smale condition on $\cN_{\eps}(\o)$ and is bounded from below there. Using \cite[Theorem 3.3]{csz} again, it follows that the same is true of $\Psi_\eps$ on $\cU_\eps$. 
 
Note that $\cU_\eps$ and $\Psi_\eps$ are $\cZ$-invariant and the action of $\cZ$ on $\cT$ is free, that is, the $\cZ$-orbit of every point $\bf u\in\cT$ is $\cZ$-homeomorphic to $\cZ$. So, by \cite[Theorem 1.1] {cp2}, $\Psi_\eps$ has at least $\mathscr{G}$-$\cat(\cU_\eps^{\leq d})$ critical $\cZ$-orbits in 
$$\cU_\eps^{\leq d}:=\{\bf u\in\cU_{\eps}:\Psi_\eps(\bf u)\leq d\},$$
where $\mathscr{G}:=\{\cZ\}$. Since $\cZ$ acts freely on $\cU_\eps^{\leq d}$ one easily verifies that  
$$\mathscr{G}\text{-}\cat(\cU_\eps^{\leq d})=\cat(\cU_\eps^{\leq d}/\cZ),$$
and, as $\bf m_\eps$ induces a homeomorphism of the $\cZ$-orbit spaces, one has that $\cat(\cU_\eps^{\leq d}/\cZ)=\cat(\cN_\eps^{\leq d}/\cZ)$. It follows that $\cJ_\eps$ has at least $\cat\left(\cN_{\eps}(\o)^{\leq d}/\cZ\right)$
critical $\cZ$-orbits in $\cN_{\eps}(\o)^{\leq d}$, as claimed.
\end{proof}

\section{The sign of low energy solutions}\label{sec:sign}	
	
For each $i=1,\ldots,\ell$, consider the problem
\begin{equation} \label{eq:limit_problem}
\begin{cases}
-\Delta w+w=\beta_{ii}|w|^{2p-2}w,\\
w\in H^1(\rn), \quad w\neq 0.
\end{cases}
\end{equation}
Its solutions are the critical points of the functional $J_{\infty,i}:H^1(\rn)\to\r$ given by
$$J_{\infty,i}(w):=\frac{1}{2}\irn(|\nabla w|^2+w^2)-\frac{1}{2p}\irn\beta_{ii}|w|^{2p}$$
on the Nehari manifold \ $\cM_{\infty,i}:=\{w\in H^1(\rn):w\neq 0, \ J_{\infty,i}'(w)w=0\}$. \ We set
$$\mathfrak{c}_{\infty,i}:= \inf_{w\in\cM_{\infty,i}}J_{\infty,i}(w).$$
A solution $\omega_i$ to this problem that satisfies $J_{\infty,i}(\omega_i)=\mathfrak{c}_{\infty,i}$ is called a \emph{least energy solution}. It is known to be radially symmetric up to translation, see \cite{bel,gnn} or \cite[Appendix C]{w}.

Recall that the energy functional for the system \eqref{eq:system} in $\rn$ with $\eps=1$ is $\cJ_1:H^1(\rn)^\ell\to\r$ given by
\begin{align*}
\cJ_1(\bf u):= \frac{1}{2}\sum_{i=1}^\ell\irn\left(|\nabla u_i|^2 + u_i^2\right) - \frac{1}{2p}\sum_{i,j=1}^\ell\irn\beta_{ij}|u_j|^p|u_i|^p,
\end{align*}
the Nehari set is \ $\cN_{1}(\mathbb R^N):=\Big\{\bf u\in H^1(\mathbb R^N)^\ell:u_i\neq 0, \ \partial_i\cJ_1(\bf u)u_i=0, \ \forall \ i = 1,\ldots,\ell\Big\}$, and
\begin{equation*}
c_1(\rn):= \inf_{\bf u\in\cN_1(\rn)}\cJ_1(\bf u).
\end{equation*}

\begin{lemma} \label{lem:limit_energy}
$c_1(\rn)=\sum_{i=1}^\ell \mathfrak{c}_{\infty,i}=\displaystyle\lim_{\eps\to 0}c_\eps(\o)$, and $c_1(\rn)$ is not attained if $\ell\geq 2$. 
\end{lemma}

\begin{proof}
These are the statements of \cite[Proposition 3.1 and Lemma 4.2]{cs} with $G$ equal to the trivial group. Note that for such $G$ the set $\mathrm{Fix}(G) := \{x\in\rn: gx=x \text{ for all }g\in G\}$ has positive dimension (in fact, $\mathrm{Fix}(G) = \rn$). This is needed in order to assure that $c_1(\rn)$ equals the sum above.
\end{proof}

The following result describes the behavior of minimizing sequences for the system \eqref{eq:system} in $\rn$ with $\eps=1$. We write $B_1(x)$ for the unit ball in $\rn$ centered at $x$.

\begin{theorem} \label{thm:profiles}
 Let $\ell \geq 2$ and $\mathbf{w}_k=(w_{1 k}, \ldots, w_{\ell k}) \in \mathcal{N}_{1}(\mathbb R^N)$ be such that $\cJ_1(\bf w_k)\to c_1(\rn)$. Then, there exist a positive number $\theta$ and, for each $i=1, \ldots, \ell$, a sequence $\left(\zeta_{i k}\right)$ in $\rn$ and a least energy radial solution $\omega_i$ to the problem \eqref{eq:limit_problem} such that, after passing to a subsequence,
\begin{itemize}
\item[$(i)$] $\int_{B_1(\zeta_{ik})}\beta_{ii}|w_{ik}|^{2p}\geq\theta>0,$
\item[$(ii)$] $\lim _{k\to\infty}\left|\zeta_{i k}-\zeta_{j k}\right|=\infty$ \ if \ $i\neq j$,
\item[$(iii)$] $\lim _{k \rightarrow \infty}\left\|w_{i k}-\omega_i\left(\cdot-\zeta_{i k}\right)\right\|_1=0$,
\end{itemize}
for $i,j=1,\ldots,\ell$. Furthermore, if $w_{i k} \geq 0$ for all $k \in \mathbb{N}$, then $\omega_i \geq 0$.
\end{theorem}

\begin{proof}
This is the statement of \cite[Theorem 3.3]{cs} when $G$ is the trivial group. The inequality $(i)$ is marked as (3.2) in the proof of that result.
\end{proof}

\begin{proposition} \label{prop:sign}
There exist $\eps_0>0$ and $d>c_1(\rn)$ such that, if $\eps\in(0,\eps_0)$ and $\bf u=(u_1,\ldots,u_\ell)$ is a critical point of $\cJ_\eps$ with $\cJ_\eps(\bf u)\leq d$, then $u_i$ does not change sign for any $i=1,\ldots,\ell$.
\end{proposition}

\begin{proof}
We argue by contradiction. Assume there are numbers $\eps_k>0$ and $d_k>c_1(\rn)$ and critical points $\bf u_k=(u_{1 k}, \ldots, u_{\ell k})$ of $\cJ_{\eps_k}$ such that $\eps_k\to 0$, $d_k\to c_1(\rn)$, $\bf u_k\in\cN_{\eps_k}(\o)$,  $\cJ_{\eps_k}(\bf u_k)\leq d_k$, and $u_{ik}^+:=\max\{u_{ik},0\}\neq 0$ and $u_{ik}^-:=\min\{u_{ik},0\}\neq 0$ for some $i=1,\ldots,\ell$. Without loss of generality, we assume that $i=1$.

Let $\bf w_k=(w_{1 k}, \ldots, w_{\ell k})$ be given by $w_{ik}(z):=u_{ik}(\eps_kz)$ if $z\in\o_k:=\{\eps_kx:x\in\o\}$ and $w_{ik}(z):=0$ if $z\in\rn\smallsetminus\o_k$. Then $\bf w_k\in\cN_1(\rn)$ and $\cJ_1(\bf w_k)=\cJ_{\eps_k}(\bf u_k)\to c_1(\rn)$. By Lemma~\ref{lem:limit_energy} and Theorem~\ref{thm:profiles},
\begin{align*}
J_{\infty,i}(w_{ik}) = J_{\infty,i}(w_{ik}(\cdot +\zeta_{ik})) \to J_{\infty,i}(\omega_{i})=\mathfrak{c}_{\infty,i}\qquad \text{ and }\qquad
\cJ_1(\bf w_k)\to c_1(\rn)=\displaystyle\sum_{i=1}^\ell \mathfrak{c}_{\infty,i},
\end{align*}
where $\omega_i$ is a least energy solution to the problem \eqref{eq:limit_problem}. In particular,
\begin{align*}
\lim_{k\to\infty}\sum\limits_{\substack{i,j=1 \\ i\neq j}}^\ell\int_{\rn}\beta_{ij}|w_{ik}|^p|w_{jk}|^p=2p
\lim_{k\to\infty}\left(\sum_{i=1}^\ell J_{\infty,i}(w_{ik})-\cJ_1(\bf w_k)\right)=0.
\end{align*}
Since $\bf u_k$ is a critical point of $\cJ_{\eps_k}$, one has that $\partial_1\cJ_1(\bf w_k)w_{1k}^+=\partial_1\cJ_{\eps_k}(\bf u_k)u_{1k}^+=0$. Hence
 \begin{align} \label{neh}
\int_{\rn}(|\nabla w_{1k}^+|^2 + |w_{1k}^+|^2-\beta_{11}|w_{1k}^+|^{2p})=\sum_{j=2}^\ell\int_{\rn}\beta_{1j}|w_{jk}|^p|w_{1k}^+|^{p}\to 0,
 \end{align}
and similarly for $w_{1k}^-$ (that the right-hand side above tends to 0 follows from (ii) and (iii) of Theorem \ref{thm:profiles}). Set
$$t_k:=\left(\frac{\irn(|\nabla w_{1k}^+|^2 + |w_{1k}^+|^2)}{\irn\beta_{11}|w_{1k}^+|^{2p}}\right)^\frac{1}{2p-2}\text{ \ and \ }s_k:=\left(\frac{\irn(|\nabla w_{1k}^-|^2 + |w_{1k}^-|^2)}{\irn\beta_{11}|w_{1k}^-|^{2p}}\right)^\frac{1}{2p-2}.$$
Then $t_k w_{1k}^+,s_k w_{1k}^-\in \cM_{\infty,1}$, $t_k\to 1$ and $s_k\to 1$. Here we have used \eqref{neh} and the fact that $(w_{1k}^\pm,w_{2k},\ldots,w_{\ell k})\in \cN_1(\rn)$ which implies in particular that $w_{1k}^\pm$ are bounded away from 0 \cite[Proposition 3.1(d)]{csz}.
 Thus, by the definition of $\mathfrak{c}_{\infty,1}$,
 \begin{align*}
\cJ_1(\bf w_k)&=\sum_{i=1}^\ell \mathfrak{c}_{\infty,i}+o(1)=\sum_{i=1}^\ell J_{\infty,i}(w_{ik})+o(1)=J_{\infty,1}(w_{1k}^+)+J_{\infty,1}(w_{1k}^-)+\sum_{i=2}^\ell J_{\infty,i}(w_{ik})+o(1)\\
  &=J_{\infty,1}(t_k w_{1k}^+)+J_{\infty,1}(s_k w_{1k}^-)+\sum_{i=2}^\ell J_{\infty,i}(w_{ik})+o(1)\geq 2\mathfrak{c}_{\infty,1}+\sum_{i=2}^\ell \mathfrak{c}_{\infty,i}+o(1)\\
  &=\mathfrak{c}_{\infty,1}+\cJ_1(\bf w_k)+o(1),
 \end{align*}
which is a contradiction.
\end{proof}

\section{The effect of the configuration space}\label{sec:cs}

Fix $r\in(0,\infty)$ and let $B_r$ be the ball of radius $r$ in $\rn$ centered at $0$. For each $i=1,\ldots,\ell$, consider the problem
\begin{equation} \label{eq:B_r}
\begin{cases}
-\eps^2\Delta w+w=\beta_{ii}|w|^{2p-2}w,\\
w\in H^1_0(B_r), \quad w\neq 0.
\end{cases}
\end{equation}
Its solutions are the critical points of the functional $J_{\eps,r,i}:H^1_0(B_r)\to\r$ given by
$$J_{\eps,r,i}(w):=\frac{1}{2\eps^N}\int_{B_r}(\eps^2|\nabla w|^2+w^2)-\frac{1}{2p\eps^N}\int_{B_r}\beta_{ii}|w|^{2p}$$
on the Nehari manifold $\cM_{\eps,r,i}:=\{w\in H^1_0(B_r):w\neq 0, \ J_{\eps,r,i}'(w)w=0\}$. We set
$$\mathfrak{c}_{\eps,r,i}:= \inf_{w\in\cM_{\B{\eps,}r,i}}J_{\eps,r,i}(w).$$
A solution $w_{\eps,r,i}$ satisfying $J_{\eps,r,i}(w_{\eps,r,i})=\mathfrak{c}_{\eps,r,i}$ is called a \emph{least energy solution}.

\begin{lemma} \label{lem:limit_energy2}
$\displaystyle\lim_{\eps\to 0}\mathfrak{c}_{\eps,r,i} =\displaystyle\mathfrak{c}_{\infty,i}$ for each $i=1,\ldots,\ell$. 
\end{lemma}

\begin{proof}
This follows from Lemma~\ref{lem:limit_energy}, taking $\o=B_r$ and $\ell=1$.
\end{proof}

For $r>0$ let
$$F_{\ell,r}(\o):=\{(\xi_1,\ldots,\xi_\ell)\in\o^\ell:\dist(\xi_i,\rn\smallsetminus\o)> r, \ |\xi_i-\xi_j|> 2r\text{ if }i\neq j\},$$
endowed with the subspace topology of $\o^\ell$. Note that $F_{\ell,r}(\o)\neq\emptyset$ if $r$ is sufficiently small. Define $\bf i_\eps:F_{\ell,r}(\o)\to \cN_{\eps}(\o)$ by
\begin{align}\label{i:def}
\bf i_\eps(\xi_1,\ldots,\xi_\ell):=\left(\omega_{\eps,r,1}(\,\cdot\,-\xi_1),\ldots,\omega_{\eps,r,\ell}(\,\cdot\,-\xi_\ell)\right)
\end{align}
where $\omega_{\eps,r,i}$ is the positive least energy solution to the problem \eqref{eq:B_r} extended by 0 to $\rn\smallsetminus\o$. 

\begin{lemma} \label{lem:iota}
Given $r>0$ and $d>\sum_{i=1}^\ell \mathfrak{c}_{\infty,i}$, there exists $\widetilde\eps_{r,d}>0$ such that
$$\cJ_\eps(\bf i_\eps(\xi_1,\ldots,\xi_\ell))\leq d\qquad\text{for every \ }(\xi_1,\ldots,\xi_\ell)\in F_{\ell,r}(\o)\text{ \ and \ }\eps\in (0,\widetilde\eps_{r,d}).$$
\end{lemma}

\begin{proof}
Let $d>\sum_{i=1}^\ell \mathfrak{c}_{\infty,i}$ and $(\xi_1,\ldots,\xi_\ell)\in F_{\ell,r}(\o)$. For $\eps>0$ set $u_i(x):=\omega_{\eps,r,i}(x-\xi_i)$. Then, since the components have disjoint supports,
\begin{align*}
 \cJ_\eps(\bf i_\eps(\xi_1,\ldots,\xi_\ell))
 &= \frac{1}{2\eps^N}\sum_{i=1}^\ell\io\left(\eps^2|\nabla u_i|^2 + u_i^2\right) - \frac{1}{2p\eps^N}\sum_{i,j=1}^\ell\io\beta_{ij}|u_j|^p|u_i|^p\\
 &= \frac{1}{2\eps^N}\sum_{i=1}^\ell\io\left(\eps^2|\nabla u_i|^2 + u_i^2\right) - \frac{1}{2p\eps^N}\beta_{ii}|u_i|^{2p}\\
 &=\sum_{i=1}^\ell
 J_{\eps,r,i}(\omega_{\eps,r,i}) = \sum_{i=1}^\ell\mathfrak{c}_{\eps,r,i}
 \to \sum_{i=1}^\ell \mathfrak{c}_{\infty,i}
\end{align*}
as $\eps\to0$ by Lemma~\ref{lem:limit_energy2}, and the claim follows.
\end{proof}

Let $b:L^2(\rn)\smallsetminus\{0\}\to\rn$ be a \emph{generalized barycenter map} as constructed in \cite[Sec. 2]{bw}. It is shown in \cite[(2.3)]{bcw} that $b$ can be taken to be equivariant with respect to scaling. In other words, $b$ has the following properties: For every $u \in L^2(\rn)\smallsetminus\{0\}$,
\begin{enumerate}
\item[$(B_1)$] $b(u)=b(|u|)$,
\item[$(B_2)$] If $\xi \in \mathbb{R}^N$ and $u_{\xi}(x):=u(x-\xi)$, then $b(u_{\xi})=b(u)+\xi$,
\item[$(B_3)$] If $u$ is radially symmetric with respect to $\xi \in \mathbb{R}^N$, then $b(u)=\xi$.
\item[$(B_4)$] $\eps^{-1}b(u)=b(u(\eps \, \cdot \,))$ for every $\eps>0$.
\end{enumerate}

Let $\bf b:(L^2(\rn)\smallsetminus\{0\})^\ell\to(\rn)^\ell$ be given by
$$\bf b(\bf u):=(b(u_1),\ldots,b(u_\ell)),\quad\text{where \ }\bf u=(u_1,\ldots,u_\ell)\in(L^2(\rn)\smallsetminus\{0\})^\ell.$$
For $r>0$ set $\o_r^+:=\{x\in\rn:\dist(x,\o)\leq r\}$, and let
$$F_\ell(\o_r^+):=\{(\xi_1,\ldots,\xi_\ell)\in(\o_r^+)^\ell:\xi_i\neq \xi_j\text{ if }i\neq j\}$$
be the ordered configuration space of $\ell$ points in $\o_r^+$. 

\begin{proposition} \label{prop:b}
Given $r>0$ there exist $d_r>\sum_{i=1}^\ell \mathfrak{c}_{\infty,i}$ and $\widehat{\eps}_r>0$ such that $d_r>c_\eps(\o)$ and
$$
\bf b(\bf u)\in F_\ell(\o_r^+)\qquad\text{for every \ }\bf u\in\cN_{\eps}(\o)^{\leq d_r}\text{ \ with \ }\eps\in (0,\widehat{\eps}_r).
$$
\end{proposition}

\begin{proof}
We argue by contradiction. Fix $r>0$ and assume there are \ $d_k>\sum_{i=1}^\ell \mathfrak{c}_{\infty,i}$, \ $\eps_k>0$ \ and \ $\bf u_k=(u_{1 k}, \ldots, u_{\ell k})\in\cN_{\eps_k}(\o)^{\leq d_r}$ such that $d_k\to \sum_{i=1}^\ell \mathfrak{c}_{\infty,i}$ and $\eps_k\to 0$ as $k\to\infty$, and
$$\bf b(\bf u_k)\not\in F_\ell(\o_r^+).$$
Then, passing to a subsequence, there are two possibilities:
\begin{itemize}
\item[$(I)$] either there exist $i\neq j$ such that $b(u_{ik})=b(u_{jk})$ for every $k$,
\item[$(II)$] or there exists $i$ such that $\dist(b(u_{ik}),\o)\ge r$ for all $k$.
\end{itemize}
Next we prove that this leads to a contradiction.

Let $\bf w_k=(w_{1 k}, \ldots, w_{\ell k})$ be given by $w_{ik}(z):=u_{ik}(\eps_kz)$ if $z\in\o_k:=\{\eps_kx:x\in\o\}$ and $w_{ik}(z):=0$ if $z\in\rn\smallsetminus\o_k$. Then $\bf w_k\in\cN_1(\rn)$ and $\cJ_1(\bf w_k)=\cJ_{\eps_k}(\bf u_k)\to \sum_{i=1}^\ell \mathfrak{c}_{\infty,i}=c_1(\rn)$. So, for each $i=1, \ldots, \ell$, there exist a sequence $\left(\zeta_{i k}\right)$ in $\rn$ and a least energy radial solution $\omega_i$ to the problem \eqref{eq:limit_problem} satisfying statements $(i)$, $(ii)$ and $(iii)$ of Theorem~\ref{thm:profiles}. Statement $(iii)$ yields
$$w_{ik}(\,\cdot \, +\zeta_{ik})\to\omega_i\quad\text{in \ }H^1(\rn).$$
Set $\xi_{ik}:=\eps_k\zeta_{ik}$. From properties $(B_2)$, $(B_3)$ and $(B_4)$ and the continuity of the barycenter map we derive
\begin{equation} \label{eq:barycenter}
\eps_k^{-1}\left(b(u_{ik})-\xi_{ik}\right)=b(w_{ik})-\zeta_{ik}=b(w_{ik}(\,\cdot \, +\zeta_{ik}))\to b(\omega_i)=0.
\end{equation}
This implies that $(I)$ cannot hold true, otherwise
$$|\zeta_{ik}-\zeta_{jk}|=\eps_k^{-1}|\xi_{ik}-\xi_{jk}|\leq \eps_k^{-1}|\xi_{ik}-b(u_{ik})| + \eps_k^{-1}|b(u_{jk})-\xi_{jk}|\to 0,$$
contradicting statement $(ii)$ of Theorem~\ref{thm:profiles}. Furthermore, statement $(i)$ implies that $\dist(\zeta_{ik},\o_k)\leq 1$. Then, \eqref{eq:barycenter} yields that $\dist(b(w_{ik}),\o_k)\leq 2$ for large enough $k$ and, using property $(B_4)$, we get that
$$\dist(b(u_{ik}),\o)=\eps_k\dist(b(w_{ik}),\o_k)\leq 2\eps_k\to 0.$$
This contradicts $(II)$ and completes the proof.
\end{proof}

\begin{proof}[Proof of Theorem~\ref{thm:main}] 
Let $\eps_0>0$ and $d>\sum_{i=1}^\ell \mathfrak{c}_{\infty,i}$ be as in Proposition~\ref{prop:sign} and, for the given $r$, let $d_r>\sum_{i=1}^\ell \mathfrak{c}_{\infty,i}$ and $\widehat{\eps}_r>0$ be as in Proposition~\ref{prop:b}. Set $\overline{d}:=\min\{d,d_r\}$ and let $\widetilde{\eps}_{r,\overline{d}}$ be as in Lemma~\ref{lem:iota}. Define $\eps_r:=\min\{\eps_0,\widetilde{\eps}_{r,\overline{d}},\widehat{\eps}_r\}$. Then, for every $\eps\in(0,\eps_r)$, the maps
$$
F_{\ell,r}(\o)\xrightarrow{\bf i_\eps} \cN_{\eps}(\o)^{\leq \overline{d}}\xrightarrow{\bf b}F_\ell(\o_r^+)
$$
are well defined and, by property $(B_3)$ of the barycenter map, their composition is the inclusion $F_{\ell,r}(\o)\hookrightarrow F_\ell(\o_r^+)$. 

Let $\cZ:=(\z_2)^\ell$ act on $H^1_0(\o)^\ell$ as stated in \eqref{eq:action}. Property $(B_1)$ yields $\bf b(\bf s\bf u)=\bf b(\bf u)$ for every $\bf s\in\cZ$, so $\bf b$ can be written as $\bf b=\widetilde{\bf b}\circ q$, where $q:\cN_{\eps}(\o)^{\leq \overline{d}}\to\cN_{\eps}(\o)^{\leq \overline{d}}/\cZ$ is the map that associates to each $\bf u\in\cN_{\eps}(\o)^{\leq \overline{d}}$ its $\cZ$-orbit $\cZ\bf u$, and $\widetilde{\bf b}(\cZ\bf u):=\bf b(\bf u)$. Therefore, the composition
$$F_{\ell,r}(\o)\xrightarrow{\bf i_\eps}\cN_{\eps}(\o)^{\leq \overline{d}}\xrightarrow{q}\cN_{\eps}(\o)^{\leq \overline{d}}/\cZ\xrightarrow{\widetilde{\bf b}}F_\ell(\o_r^+)$$
equals $F_{\ell,r}(\o)\hookrightarrow F_\ell(\o_r^+)$. From Theorem~\ref{thm:cat}, Proposition~\ref{prop:sign} and property $(C_1)$ in Section~\ref{sec:configuration} we derive that $\cJ_\eps$ has at least
$$\cat(\cN_{\eps}(\o)^{\leq \overline{d}}/\cZ)\geq\cat(F_{\ell,r}(\o)\hookrightarrow F_\ell(\o_r^+))$$
critical $\cZ$-orbits in $\cN_{\eps}(\o)^{\leq\overline{d}}$ whose components do not change sign for any $\eps\in(0,\eps_r)$.

Observe that, if the components of $\bf u=(u_1,\ldots,u_\ell)$ do not change sign, then there exists a unique $\bf s\in\cZ$ such that $\bf s\bf u$ is nonnegative. This completes the proof.
\end{proof}

\section{The category of configuration spaces}
\label{sec:configuration}

Let $X$ and $Y$ be topological spaces and $f:X\to Y$ be continuous. The \emph{category of $f$}, denoted $\cat(f)$, is the smallest number of open subsets of $X$ that cover $X$  and have the property that the image under $f$ of each of them is contractible in $Y$. The \emph{category of $X$} is the category of the identity map \ $\mathrm{id}_X:X\to X$. If the map is an inclusion, then $\cat(A\hookrightarrow X)$ is the category of $A$ in $X$, as defined in the introduction.

The following properties are easily proved, cf. \cite[Subsection 1.3]{cp}.

\begin{itemize}
\item[$(C_1)$] For any two continuous functions $f:X\to Y$ and $g:Y\to Z$,
$$\cat(g\circ f)\leq\min\{\cat(f),\cat(g)\}.$$
In particular, $\cat(f)\leq\min\{\cat(X),\cat(Y)\}$.
\item[$(C_2)$] If $f,g:X\to Y$ are homotopic, then $\cat(f)=\cat(g)$.
\end{itemize}
Observe that $(C_1)$ and $(C_2)$ imply
\begin{itemize}
\item[$(C_3)$] If $h:X\to Y$ is a homotopy equivalence, then $\cat(X)=\cat(h)=\cat(Y)$.
\item[$(C_4)$] If $h:X\to Y$ is a homotopy equivalence, $g:Y\to Z$ and $f: X\to Z$ is homotopic to $g\circ h$, then $\cat(f)=\cat(g)$.
\end{itemize}
Indeed, if $\tilde{h}:Y\to X$ is a homotopy inverse of $h$ (i.e., $\tilde{h}\circ h$ is homotopic to $\mathrm{id}_X$ and $h\circ\tilde{h}$ is homotopic to $\mathrm{id}_Y$), then $\cat(X)=\cat(\tilde{h}\circ h)\leq\cat(h)\leq\cat(Y)$ and $\cat(Y)=\cat(h\circ\tilde{h})\leq\cat(h)\leq\cat(X)$. This proves $(C_3)$. To prove $(C_4)$ note that $\cat(f)=\cat(g\circ h)\leq\cat(g)$ and, since $f\circ\tilde{h}$ is homotopic to $g$, we also have $\cat(g)=\cat(f\circ\tilde{h})\leq\cat(f)$.

Let $\o$ be a bounded smooth domain in $\rn$, $N\geq 2$, $r\geq 0$ and $\ell\geq 2$. Recall that
$$F_{\ell,r}(\o):=\{(\xi_1,\ldots,\xi_\ell)\in\o^\ell:\dist(\xi_i,\rn\smallsetminus\o)>r\text{ \ and \ }|\xi_i-\xi_j|> 2r\text{ if }i\neq j\},$$
$F_\ell(\o):=F_{\ell,0}(\o)$, \ and \ $\o_r^+:=\{x\in\rn:\dist(x,\o)<r\}$. We set $B_r(0):=\{\xi\in\rn:|\xi|<r\}$ and write $\overline{B}_r(0)$ for its closure. The following results will be used in the proof of Theorem~\ref{thm:configuration}.

\begin{lemma} \label{lem:configuration_rn}
\begin{itemize}
\item[$(i)$] $\cat(F_\ell(B_r(0))\hookrightarrow F_\ell(\rn))=\cat(F_\ell(B_r(0)))=\cat(F_\ell(\rn))=\ell$ for every $r>0$.
\item[$(ii)$] If $\o$ is convex, then $\cat(F_\ell(\o))=\ell$.
\end{itemize}
\end{lemma}

\begin{proof}
$(i):$ \ Fix a (strictly increasing) homeomorphism $h:[0,r)\to[0,\infty)$. Then, the map $\widehat{h}:F_\ell(B_r(0))\to F_\ell(\rn)$ given by
$$\widehat{h}(\xi_1,\ldots,\xi_\ell):=(h(|\xi_1|)\xi_1,\ldots,h(|\xi_\ell|)\xi_\ell)$$
is a homeomorphism. The map \ $F_\ell(B_r(0))\times[0,1]\to F_\ell(\rn)$ \ defined by
$$(\xi_1,\ldots,\xi_\ell,t)\mapsto ((1-t)\xi_1+th(|\xi_1|)\xi_1,\ldots,(1-t)\xi_\ell+th(|\xi_\ell|)\xi_\ell)$$
is a homotopy between the inclusion $F_\ell(B_r(0))\hookrightarrow F_\ell(\rn)$ and $\widehat{h}$. So from properties $(C_2)$ and $(C_3)$ we get
$$\cat(F_\ell(B_r(0))\hookrightarrow F_\ell(\rn))=\cat(\widehat{h})=\cat(F_\ell(B_r(0)))=\cat(F_\ell(\rn)).$$
It is shown in \cite[Theorem 1.2]{r} that $\cat(F_\ell(\rn))=\ell$. (Note that the definition of category used in \cite{r} equals our $\cat(X)-1$.)

$(ii):$ \ As every nonempty open convex subset of $\rn$ is homeomorphic to $\rn$ \cite[Chapter 3, Exercise J]{k}, this statement follows from \cite[Theorem 1.2]{r}.
\end{proof}

Next, we define
$$E_{\ell,r}(\o):=\{(\xi_1,\ldots,\xi_\ell)\in\o^\ell:\dist(\xi_i,\rn\smallsetminus\o)>2^{\ell-i+1}r\text{ \ and \ }|\xi_i-\xi_j|> 2^{\ell-i+1}r\text{ for all }j<i\}.$$
Note that $E_{\ell,0}(\o)=F_{\ell}(\o)$.

\begin{proposition} \label{prop:configuration}
If $\o$ is convex, then the inclusion $E_{\ell,r}(\o)\hookrightarrow F_\ell(\o)$ is a homotopy equivalence for small enough $r$.
\end{proposition}

For the proof of this proposition we need the following lemma which, for $r=0$, is a special case of \cite[Theorem 3]{fn}. The definition of fiber bundle may be found, for instance, in \cite[Chapter 2, Section 7]{s}.

\begin{lemma} \label{lem:configuration}
Let $\ell\geq 3$. Fix $\ell-1$ different points $q_1,\ldots,q_{\ell-1}\in\o$. Then, for any $r\geq 0$ sufficiently small, the map
$$\psi:E_{\ell,r}(\o)\to E_{\ell-1,2r}(\o),\qquad \psi(\xi_1,\ldots,\xi_\ell):=(\xi_1,\ldots,\xi_{\ell-1}),$$
is a fiber bundle with fiber $\o\smallsetminus\{q_1,\ldots,q_{\ell-1}\}$.
\end{lemma}

\begin{proof}
For $r\geq 0$ sufficiently small we have that $E_{\ell-1,2r}(\o)\neq\emptyset$ and $\o_{2r}^-\smallsetminus\{q_1,\ldots,q_{\ell-1}\}$ is homeomorphic to $\o\smallsetminus\{q_1,\ldots,q_{\ell-1}\}$, where $\o_{2r}^-:=\{x\in\o:\dist(x,\rn\smallsetminus\o)>2r\}$. Take any such $r$, and let $(\xi_1,\ldots,\xi_{\ell-1})\in E_{\ell-1,2r}(\o)$. Fix $\delta>0$ such that $2^{\ell-i+1}r+4\delta<\min_{j<i}|\xi_i-\xi_j|$ and $\dist(\xi_i,\rn\smallsetminus\o)>2^{\ell-i+1}r+2\delta$. Then $U:=B_{\delta}(\xi_1)\times\cdots\times B_{\delta}(\xi_{\ell-1})\subset E_{\ell-1,2r}(\o)$ and $|\xi_i-\xi_j|>4r+4\delta$ if $i\neq j$ and $ i,j\in\{1,\ldots,\ell-1\}$. For each $j$ we fix a continuous function
$$h_j:U\times \overline{B}_{2r+2\delta}(\xi_j)\to\overline{B}_{2r+2\delta}(\xi_j)$$
such that, for each $\overline{\zeta}=(\zeta_1,\ldots,\zeta_{\ell-1})\in U$, the function $h_{j,\overline{\zeta}}(x):=h_j(\overline{\zeta},x)$ satisfies
\begin{itemize}
\item $h_{j,\overline{\zeta}}(x)=x$ if $x\in\partial\overline{B}_{2r+2\delta}(\xi_j)$,
\item $h_{j,\overline{\zeta}}(x)=x-\zeta_j+\xi_j$ if $x\in\overline{B}_{2r}(\xi_j)$,
\item $h_{j,\overline{\zeta}}(x):\overline{B}_{2r+2\delta}(\xi_j)\to\overline{B}_{2r+2\delta}(\xi_j)$ is a homeomorphism.
\end{itemize}
Note that $h_{j,\overline{\zeta}}$ maps $\overline{B}_{2r}(\zeta_j)$ onto $\overline{B}_{2r}(\xi_j)$. Now define $h_{\overline{\zeta}}:\o\to\o$ by
\begin{equation*}
h_{\overline{\zeta}}(x):=
\begin{cases}
h_{j,\overline{\zeta}}(x) &\text{if \ }x\in\overline{B}_{2r+2\delta}(\xi_j), \\
x &\text{if \ }x\in\o\smallsetminus\bigcup\limits_{j=1}^{\ell-1}B_{2r+2\delta}(\xi_j).
\end{cases}
\end{equation*}
Since $\overline{B}_{2r+2\delta}(\xi_j)\subset\o_{2r}^-$, we have that $h_{\overline{\zeta}}$ maps \ $\o_{2r}^-\smallsetminus\bigcup\limits_{j=1}^{\ell-1}\overline{B}_{2r}(\zeta_j)$ \ onto \ $\o_{2r}^-\smallsetminus\bigcup\limits_{j=1}^{\ell-1}\overline{B}_{2r}(\xi_j)$. \ Note that
$$\psi^{-1}(U)=\Big\{(\zeta_1,\ldots,\zeta_{\ell}):(\zeta_1,\ldots,\zeta_{\ell-1})\in U \text{ \ and \ } \zeta_\ell\in \o_{2r}^-\smallsetminus\bigcup\limits_{j=1}^{\ell-1}\overline{B}_{2r}(\zeta_j)\Big\}.$$
Now, we fix homeomorphisms
$$f:\,\o_{2r}^-\smallsetminus\bigcup\limits_{j=1}^{\ell-1}\overline{B}_{2r}(\xi_j)\cong\o_{2r}^-\smallsetminus\{\xi_1,\ldots,\xi_{\ell-1}\}\cong\o_{2r}^-\smallsetminus\{q_1,\ldots,q_{\ell-1}\}\cong\o\smallsetminus\{q_1,\ldots,q_{\ell-1}\}.$$
For the second one see \cite[Chapter 4, Homogeneity Lemma]{m2}. As indicated above, we denote the composition of these homeomorphisms by $f$. The map \ $\tau:\psi^{-1}(U)\to U\times\left(\o\smallsetminus\{q_1,\ldots,q_{\ell-1}\}\right)$ \ given by
$$\tau(\zeta_1,\ldots,\zeta_{\ell}):=(\zeta_1,\ldots,\zeta_{\ell-1},f(h_{(\zeta_1,\ldots,\zeta_{\ell-1})}(\zeta_\ell)))$$
is a homeomorphism and satisfies $\mathrm{proj}_U\circ\tau=\psi|_U$. This proves that $\psi$ is a fiber bundle. 
\end{proof}

\begin{proof}[Proof of Proposition~\ref{prop:configuration}] Without loss of generality we assume that $0\in\o$. We prove the statement by induction on $\ell$.

Let $\ell=2$. We fix $s>0$ such that $\overline{B}_{s}(0)\subset\o_{4s}^-:=\{x\in\o:\dist(x,\rn\smallsetminus\o)>4s\}$. For each $r\in[0,s)$ we define maps
\begin{align*}
&\iota_r:\s^{N-1}\to E_{2,r}(\o),\qquad \iota_r(x):=(sx,-sx), \\
&\varrho_r:E_{2,r}\to \s^{N-1},\qquad \varrho_r(\xi_1,\xi_2)=\frac{\xi_1-\xi_2}{|\xi_1-\xi_2|},
\end{align*}
where $\s^{N-1}$ is the unit sphere in $\rn$. Then, $\varrho_r\circ\iota_r=\mathrm{id}_{\s^{N-1}}$. Furthermore, since $\o$ is convex and $s>r$, the map $E_{2,r}(\o)\times[0,1]\to E_{2,r}(\o)$ given by
$$(\xi_1,\xi_2,t)\longmapsto\left((1-t)\xi_1+ts\frac{\xi_1-\xi_2}{|\xi_1-\xi_2|},\,(1-t)\xi_2+ts\frac{\xi_2-\xi_1}{|\xi_1-\xi_2|}\right)$$
is well defined and it is a homotopy between $\mathrm{id}_{E_{2,r}(\o)}$ and $\iota_r\circ\varrho_r$. Since the composition 
$$\s^{N-1}\xrightarrow{\iota_r} E_{2,r}(\o)\hookrightarrow F_2(\o)\xrightarrow{\varrho_0}\s^{N-1}$$
is the identity map on $\s^{N-1}$ and $\iota_r$ and $\varrho_0$ are homotopy equivalences, we have that $E_{2,r}(\o)\hookrightarrow F_2(\o)$ is a homotopy equivalence.

Assume the statement holds true for $\ell-1$ with $\ell\geq 3$. For every $r\geq 0$ small enough, as in Lemma~\ref{lem:configuration}, the map $\psi:E_{\ell,r}(\o)\to E_{\ell-1,2r}(\o)$ is a fiber bundle, hence, it is a fibration \cite[Corollary 2.7.14]{s} with fiber $\o\smallsetminus Q_\ell$ where $Q_\ell:=\{q_1,\ldots,q_{\ell-1}\}$. Note that $E_{\ell,0}(\o)=F_{\ell}(\o)$. Note also that $E_{\ell,r}(\o)$ and $F_{\ell}(\o)$ are path connected for every $\ell\geq 2$. This follows by induction using the fiber bundle structure, because they are homotopic to $\s^{N-1}$ if $\ell=2$ and $N\geq 2$. From the long exact homotopy sequence of a fibration \cite[Theorem 7.2.10]{s} we get a commutative diagram
\begin{equation*}
\begin{matrix}
\pi_{i+1}(F_{\ell-1}(\o)) & \rightarrow & \pi_i(\o\smallsetminus Q_\ell) & \rightarrow & \pi_i(F_\ell(\o)) & \xrightarrow{\psi_*}  & \pi_i(F_{\ell-1}(\o)) & \rightarrow & \pi_{i-1}(\o\smallsetminus Q_\ell)\\
\uparrow & &  \uparrow & &  \uparrow & &  \uparrow & &  \uparrow \\
\pi_{i+1}(E_{\ell-1,2r}(\o)) & \rightarrow & \pi_i(\o\smallsetminus Q_\ell) & \rightarrow & \pi_i(E_{\ell,r}(\o)) & \xrightarrow{\psi_*}  & \pi_i(E_{\ell-1,2r}(\o)) & \rightarrow & \pi_{i-1}(\o\smallsetminus Q_\ell)
\end{matrix}
\end{equation*}
where all vertical arrows are induced by inclusions. The fact that the spaces above are path connected is needed in order to assert that the homotopy groups are independent of the choice of a base point.  Using the induction hypothesis we see that the two vertical leftmost arrows and the two vertical rightmost arrows in the diagram are isomorphisms for $r$ sufficiently small. So, by the five lemma (see e.g. \cite[Lemma 4.5.11]{s}), the middle vertical arrow is an isomorphism too. This proves that $E_{\ell,r}(\o)\hookrightarrow F_\ell(\o)$ is a weak homotopy equivalence. Since $E_{\ell,r}(\o)$ is an open subset of $\r^{\ell N}$, it has the homotopy type of a CW-complex \cite[Theorem 1 and Corollary 1]{m}. So, by \cite[Corollary 7.6.24]{s}, $E_{\ell,r}(\o)\hookrightarrow F_\ell(\o)$ is a homotopy equivalence. 
\end{proof}

\begin{remark}
\emph{
Note that $\pi_1(\o\smallsetminus Q_\ell)$ may not be commutative (and it never is if $N=2$ and $\ell \ge 3$). The proof of the five lemma in \cite{s} is for commutative groups. However, this lemma is known to hold also in the noncommutative case and it is easy to see that the proof in \cite{s} still applies, with obvious changes.
}
\end{remark}

\begin{proof}[Proof of Theorem~\ref{thm:configuration}] \ Without loss of generality we may assume that $0\in\o$. Fix $s>0$ such that $B_s(0)\subset\o$. Applying $(C_1)$ to the composition of inclusions \ $E_{\ell,r}(B_s(0))\hookrightarrow F_{\ell,r}(B_s(0))\hookrightarrow F_{\ell,r}(\o)\hookrightarrow F_\ell(\o_r^+)\hookrightarrow F_\ell(\rn)$ \ we get that 
$$\cat(E_{\ell,r}(B_s(0))\hookrightarrow F_\ell(\rn))\leq\cat(F_{\ell,r}(\o)\hookrightarrow F_\ell(\o_r^+)).$$
By Proposition~\ref{prop:configuration} we have that $E_{\ell,r}(B_s(0))\hookrightarrow F_{\ell}(B_s(0))$ is a homotopy equivalence for sufficiently small $r>0$. So from Lemma~\ref{lem:configuration_rn} and $(C_4)$  we get that 
$$\ell=\cat(F_{\ell}(B_s(0))\hookrightarrow F_\ell(\rn))=\cat(E_{\ell,r}(B_s(0))\hookrightarrow F_\ell(\rn)),$$
and we derive 
$$\cat(F_{\ell,r}(\o)\hookrightarrow F_\ell(\o_r^+))\geq\ell.$$

If $\o$ is convex, Proposition~\ref{prop:configuration} states that $E_{\ell,r}(\o)\hookrightarrow F_{\ell}(\o)$ is a homotopy equivalence for sufficiently small $r>0$. Note that $F_{\ell,2^{\ell}r}(\o)\subset E_{\ell,r}(\o)$. So, from $(C_1)$, $(C_2)$ and Lemma~\ref{lem:configuration_rn} we obtain that
$$\cat(F_{\ell,2^{\ell}r}(\o)\hookrightarrow F_\ell(\o_{2^{\ell}r}^+))\leq\cat(E_{\ell,r}(\o)\hookrightarrow F_\ell(\o))=\cat(F_\ell(\o))=\ell.$$
This shows that $\cat(F_{\ell,r}(\o)\hookrightarrow F_\ell(\o_r^+))=\ell$ for sufficiently small $r$ if $\o$ is convex.
\end{proof}

\subsection*{Acknowledgments}

The authors thank the referees for carefully reading the manuscript and for useful suggestions.
M. Clapp thanks the Department of Mathematics at Stockholm University for their kind hospitality. A. Saldaña and A. Szulkin thank the Instituto de Matemáticas - Campus Juriquilla for the kind hospitality. M. Clapp is supported by CONACYT (Mexico) through the research grant A1-S-10457. A. Saldaña is supported by UNAM-DGAPA-PAPIIT (Mexico) grant IA100923 and by CONACYT (Mexico) grant A1-S-10457. A. Szulkin is supported by a grant from the Magnuson foundation of the Royal Swedish Academy of Sciences.

\subsection*{Conflicts of interests}

The authors declare no conflict of interests.

\bigskip

\begin{flushleft}
\textbf{Mónica Clapp}\\
Instituto de Matemáticas\\
Universidad Nacional Autónoma de México \\
Campus Juriquilla\\
Boulevard Juriquilla 3001\\
76230 Querétaro, Qro., Mexico\\
\texttt{monica.clapp@im.unam.mx} 
\medskip 

\textbf{Alberto Saldaña}\\
Instituto de Matemáticas\\
Universidad Nacional Autónoma de México \\
Circuito Exterior, Ciudad Universitaria\\
04510 Coyoacán, Ciudad de México, Mexico\\
\texttt{alberto.saldana@im.unam.mx}
\medskip

\textbf{Andrzej Szulkin}\\
Department of Mathematics\\
Stockholm University\\
106 91 Stockholm, Sweden\\
\texttt{andrzejs@math.su.se} 
\end{flushleft}
	
\end{document}